\documentclass[a4paper]{amsart} 
\usepackage[T1]{fontenc}
\usepackage{lmodern}

\usepackage{main}
\usepackage{enumitem}

\title[Lefschetz theorem for effective cone]{Lefschetz-type theorems for the effective cone on Hyperkähler varieties}
\author[Jonas Baltes]{Jonas Baltes}
\address{Georg-August-Universität Göttingen, Mathematisches Institut, Bunsenstraße 3-5, 37073 Göttingen, Germany}
\email{jonas.baltes@mathematik.uni-goettingen.de}
\urladdr{https://sites.google.com/view/jonasbaltes}
\date{\today}

\usepackage{graphicx,pdflscape, multirow}
\DeclareMathSymbol{\shortminus}{\mathbin}{AMSa}{"39}

\newcommand{\stbdl}{\fF}

\begin{document}
\begin{abstract}
    In this paper we show some Lefschetz-type theorems for the effective cone of Hyperkähler varieties. In particular we are able to show that the inclusion of any smooth ample divisor induces an isomorphism of effective cones. Moreover we deduce a similar statement for some effective exceptional divisors, which yields the computation of the effective cone of e.g.\ projectivized cotangent bundles or some projectivized Lazarsfeld--Mukai bundles.
\end{abstract}
\maketitle
\section{Introduction}
The well known Lefschetz-hyperplane theorem states that given any ample smooth divisor $i\colon D\hookrightarrow X$ in a smooth projective variety $X$ with $\dim X \ge 4$ the restriction induces an isomorphism on the singular cohomology and the Picard group
\begin{align*}
    H^2(X, \ZZ) &\to H^2(D, \ZZ) &\Pic X &\to \Pic D.
\end{align*}

There has been some effort to examine whether similar statements also hold for more refined invariants, e.g. the effective, movable or ample cone. By the work of Hassett--Lin--Wang \cite{hassett2002weak} it turns out that a general statement cannot be valid for the nef cone. Counter examples arise already for blow ups of $\PP^4$. More recently Ottem \cite[Theorem 1.1 ii), iii)]{ottemBirationalGeometry} computed the nef, movable and effective cone of a general hypersurface $X$ of bidegree $(n,e)$ in $\PP^1\times\PP^n$ for $n\ge 3$ and it turns out all are equal and no Lefschetz-type result holds for any of these three cones, in particular
\begin{equation*}
    \Eff(X) \not\cong \Eff(\PP^1\times\PP^n)
\end{equation*}
via the inclusion.

Besides these counterexamples Hassett--Lin--Wang with the work of Kollár could prove a dual statement for the cone of curves, looking only at the $K_X$-negative part:
\begin{equation*}
    i_*\overline{NE}(D)_{K_D\le 0} \cong \overline{NE}(X)_{K_X+D\le 0}.
\end{equation*}
This solves the problem e.g. for Fano varieties but as soon as the canonical bundle is nef this does not yield any results. 

In this paper we will analyse a class of varieties, where the canonical class is indeed nef, namely we will study the \emph{effective cone} of Hyperkähler varieties, which is a class of varieties where $K_X = 0$ is trivial. We start with some background on Hyperkähler varieties in \Cref{sec:VanishingThms}.

In \Cref{sec:EffConeAmple} we prove a Lefschetz-type theorem for the effective cones of smooth complete intersections in the case that the birational geometry of such spaces is limited. It arises as a simple consequence of a theorem of Verbitsky:
\begin{theorem}
\label{mainthm1}
    Let $X$ be a Hyperkähler variety of dimension $2n\ge 4$ and $Y\subset X$ a smooth complete intersection of ample divisors of codimension $\codim_X Y = c < n$. Suppose furthermore that any smooth birational $K$-trivial model $X'$ of $X$ is isomorphic to $X$ in codimension $c+1$. Then the restriction morphism $\Pic(X)\to \Pic(Y)$ induces a bijection
    \begin{equation*}
        \Eff(X)\cong \Eff(Y).
    \end{equation*}
\end{theorem}
In \Cref{sec:EffConeExc} we want to get rid of the assumption that the smooth subvariety is a complete intersection. 
However, an isomorphism as above cannot hold in the same way: Let $S$ be a K3 surface with an ample line bundle $H$ and let $\PP(\Omega_S) = E\subset S^{[2]}$ be the exceptional divisor of the Hilbert-Chow morphism $S^{[2]}\to \Sym^2 S$, where $S^{[2]}$ is the Hilbert scheme of two points. 
An elementary argument shows, that the natural line bundle $H^{[2]}\in \Pic S^{[2]}$ is big and nef, i.e. in the interior of the cone of exceptional divisors. 
On the other hand $H^{[2]}|_E$ is the pullback of $2H$ under the contraction morphism $E = \PP(\Omega_S) \to S$ and thus, $H^{[2]}|_E$ cannot be in the interior of $\Eff(E)$.\par
By considering the Beauville-Bogomolov-Fujiki-form $\bbf{-}{-}\colon \Pic(S^{[2]})\times \Pic(S^{[2]}) \to \ZZ$ we are able to prove an analogous statement though. If we denote by
\begin{equation*}
    \Eff(S^{[2]})_{E\ge 0} = \{D\in N^1(S^{[2]}) \,|\, \bbf{D}{E}\ge 0\},
\end{equation*} 
the following holds:
\begin{theorem}
    \label{mainthm2}
    Let $S$ be a K3 surface and $X = S^{[2]}$ the Hilbert-scheme. Suppose that $X$ does not admit any flops. Then the Hilbert-Chow exceptional divisor $E$ satisfies 
     \begin{equation*}
        \Eff(E) = \Eff(X)_{E\ge 0}|_E.
    \end{equation*}
\end{theorem}
With this theorem we know how positive the cotangent bundle is, i.e. for any ample $H\in\Pic S$ this computes $\inf \{\frac{m}{n}\,|\, H^0(\Sym^n(\Omega_S)\otimes \oO(mH))\neq 0\}$.\par 
By the work of Bayer--Macr\`{i} \cite{bayer2014mmp}, the existence of flops is a purely numerical question and thus this directly gives a generalization of the results for Picard rank $1$ of Gounelas--Ottem \cite{ottem2020remarks}, as this yields the effective cone of $\PP(\Omega_S)\cong E$ in more degrees $d = H.H$ and also in higher Picard ranks.  Moreover, for Picard rank $1$ we now at least know one of the ample or effective cone in all degrees.\par
By the recent paper of Anella--Hoering \cite{anella2022cotangent} the condition that there is no flop is indeed necessary. They construct an effective divisor in $\PP(\Omega_S)$ for a K3 surface $S$ of degree $2$, which is not contained in $\Eff(X)_{E\ge 0}|_E$. They do so, by using the geometry of the fibration $S\to \PP^2$.\par
Generalizing the above further to different exceptional subvarieties we are able to compute the effective cone of some other geometrically interesting projectivized bundles. An example is that of stable Lazarsfeld--Mukai bundles $N = N_{g,r,d}$ with $\rho(g,r,d) = 0$ on a K3 surface $S$ with Picard rank $1$, which play an important role in Brill--Noether theory. By the classical theory on moduli spaces of sheaves there is an embedding $\PP(N_{g,r,d})\hookrightarrow X_N$ into a Hyperkähler variety $X_N$ such that the image is contained in the exceptional set of a contraction $X_N\to \overline{X}_N$. It turns out that in this setting a Lefschetz-type theorem holds as well:
\begin{theorem}
\label{mainthm3}
    Let $S$ be a K3 surface with Picard rank $1$ and $N$ a stable Lazarsfeld--Mukai bundle with $\rho = 0$. Suppose further that the Hyperkähler variety $X_N$ as above has no other flopping contraction except for possibly $X_N\to \overline{X}_N$. Then there is a $D\in \Pic(X_N)$ such that
    \begin{equation*}
        \Eff(\PP(N)) = \Eff(X_N)_{D\ge 0}|_{\PP(N)}.
    \end{equation*}
\end{theorem}
The paper ends with the computation of the exact numerics of effective cones. We use the work of Bayer--Macr\`{i} \cite{bayer2014mmp} to compute the effective cone of the respective Hyperkähler varieties and then use the above theorems. We discuss this in detail for the cotangent bundle and Lazarsfeld--Mukai bundles of rank $2$. Some examples are given for higher Lazarsfeld--Mukai bundles of higher rank as well.

\section{Facts on Hyperkähler varieties and examples}
\label{sec:VanishingThms}
In this section we recall the geometry of Hyperkähler varieties and discuss some examples.
\begin{definition}
    A \emph{Hyperkähler variety} is a simply connected, projective Kähler manifold $X$, such that $H^0(X, \Omega^2_X) = \CC\sigma$ is spanned by an everywhere non-degenerate 2-form $\sigma$.
\end{definition}
In dimension $2$ Hyperkähler varieties are K3 surfaces. Moreover in higher dimensions $\dim X = 2n$ the behaviour is quite similar to that low-dimensional case. For example there is the \emph{Beauville-Bogomolov-Fujiki form} (BBF-form) 
\begin{equation*}
    \bbf{-}{-}\colon H^2(X,\ZZ)\times H^2(X,\ZZ) \to \ZZ
\end{equation*}
such that there exists a constant $c_X\in \ZZ$ with
\begin{equation*}
    \int \beta^{2n} = c_X \bbf{\beta}{\beta}^n
\end{equation*}
for all $\beta\in H^2(X,\ZZ)$. With respect to this form Boucksom  \cite[Section 4]{boucksom2004divisorial} defined a Zariski-decomposition as follows:
\begin{definition}
    Let $X$ be a Hyperkähler variety and $E$ prime divisor. Then $E$ is called \emph{exceptional} if $\langle E,E\rangle < 0$. 
    A finite collection of prime divisors $E_i$ is called an \emph{exceptional family} if the corresponding matrix $(\langle E_i, E_j\rangle)_{ij}$ is negative definite.
    
    On the other hand a divisor $D$ is called \emph{movable} if $D$ has base locus of codimension at least $2$. The closure of the cone generated by movable divisors is denoted $\Mov(X)$.
\end{definition}
\begin{theorem}[{\cite[Theorem 4.8 and Corollary 4.11]{boucksom2004divisorial}}]
    \label{thm:ZariskiDecomposition}
    Let $D\in \Pic X$ be a pseudo-effective divisor. Then there exists a rational Zariski-decomposition $D = M+\sum a_i E_i$ with $a_i \in \QQ_{\ge 0}$, where $M\in \Mov (X)$ and the $E_i$ are a exceptional family such that $\langle M, E_i\rangle = 0$ for all $i$.
\end{theorem}
\begin{corollary}
\label{cor:EffMovDual}
    The effective cone $\Eff(X)$ is dual to the movable cone $\Mov(X)$ with respect to the Beauville-Bogomolov-Fujiki form.
\end{corollary}
In the next section we will use this description of the effective cone together with the vanishing theorems of Verbitsky to compute the effective cone of subvarieties.
Denote by $\overline{\kK}$ the Kähler cone and we define the dual 
\begin{equation*}
    \overline{\kK}^\vee = \{x\in \hh^{1,1}(X,\RR) \,|\, \bbf{x}{y}\ge 0 \; \text{for all}\; y\in \overline{\kK}\}.
\end{equation*}
\begin{theorem}[Verbitsky \cite{verbitsky2007quaternionic}]
\label{thm:Verbitsky}
    Let $X$ be a projective Hyperkähler variety of dimension $2n$ and let $L\in \Pic(X)$ such that $c_1(L) \notin \overline{\kK}^\vee$. Then     
    \begin{equation*}
        \hh^i(X, L) = 0
    \end{equation*}
    for $i< n$. Moreover for any vector bundle $F$ we have that there exists an $N_0\in \NN$ such that for all $N\ge N_0$ we have
    \begin{equation*}
        \hh^i(X, L^N\otimes F) = 0
    \end{equation*}
    for all $i< n$.
\end{theorem}

\section{Effective cone of smooth subvarieties}
\label{sec:EffConeAmple}
Let $X$ be a Hyperkähler variety of dimension $2n$ and $Y\subset X$ a smooth subvariety. In this section we will examine cases where the restriction map $\Pic(Y)\to \Pic(X)$ induces a bijection of the effective cones.
\begin{remark}
\label{rem:MovableDivisorsAmple}
    Recall from \cite[Prop. 1]{matsushita2014almost} that for any movable divisor $M\in \Pic(X)$ there is another Hyperkähler variety $X'$, an ample divisor $A\in \Pic(X')$ and a birational map $f\colon X\dashrightarrow X'$ such that $M = f^*A$.
\end{remark}
\begin{proposition}
\label{prop:NotEffective}
    Let $Y\subset X$ be a smooth subvariety of codimension $c < n$ and $D\in \Pic(X)$ be a divisor which is not $\QQ$-effective. Suppose there is a divisor $M$ with $\bbf{M}{D}<0$ such that either
    \begin{itemize}
        \item $M$ is ample or 
        \item  $Y$ is a complete intersection and $M$ is movable as a pullback of an ample divisor $A$ from another Hyperkähler variety $X'$ such that $X$ and $X'$ are isomorphic in codimension $c+1$.
    \end{itemize}
    Then $D|_Y$ is not $\QQ$-effective as well.
\end{proposition}
Ideally, we would want to apply Verbitsky's vanishing theorem to the ideal sheaf $I_Y\otimes D$. But that is not directly possible as $I_Y$ is not a vector bundle if $Y$ has codimension $\ge 2$. Therefore we need a nice locally free resolution of $I_Y$, which in the case that $Y$ is a complete intersection is possible by considering the Koszul complex: If $Y = \bigcap_{i=1}^c D_i$ is the intersection of divisors $V(s_i) = D_i \subset X$, where $s_i\in H^0(X, \oO(D_i))$, then the Koszul-complex associated to $$E= \bigoplus \oO_X(-D_i)\xrightarrow{(s_1,\ldots, s_n)} \oO_X$$ yields a locally free resolution of $I_Y$ of length $c-1$.\par
A resolution by vector bundles of the same length can even be found in the broader case when $Y$ is a smooth subvariety. We include a proof here as we could not find a direct reference.
\begin{lemma}
\label{lem:Resolution}
    Let $X$ be a smooth variety and $Y\subset X$ a smooth subvariety of codimension $\codim_X Y = c$. Then there exists a locally free resolution $F_{c-1} \to \ldots \to F_0 \to I_Y$ of length $c-1$.
\end{lemma}
\begin{proof}
    By \cite[Ex.\ III.6.5]{HartshorneBook} it suffices to show that the projective dimension of all stalks $I_{Y,x}$ is at most $c-1$ for all points $x\in X$ and thus only for $x\in Y$.
    Thus, let $R= \oO_{X,x}$ be a regular local ring of dimension $n$ and $I$ an ideal such that $R/I$ is regular as well and of codimension $c> 0$. Denote by $\mm\subset R$ the maximal ideal. Then it is known that $\depth_\mm R = n$ as any regular local ring is Cohen-Macaulay. But as the depth of an $R$-module is the maximal length of an $\mm$-regular sequence we have that the depth of $R/I$ as an $R$-module coincides with the depth as an $R/I$-module and thus  $\depth_\mm R/I = \depth R/I= \dim R/I = n-c$. Therefore, by the short exact sequence
    \begin{equation*}
        0 \to I \to R \to R/I \to 0
    \end{equation*}
    we get the inequalities $n-c= \depth R/I \ge \min (n, \depth I - 1)$ and also $\depth I \ge \min(n, \depth R/I+1) = n-c+1$ as $I,R/I$ are both non-zero. But this yields $\depth I = n-c+1$ and the Auslander-Buchsbaum formula gives
    \begin{equation*}
        \pdim I = \depth R - \depth I =  c-1.\qedhere
    \end{equation*}
\end{proof}
To recover the cohomology of the ideal sheaf from this data, we observe that the following holds by chasing cohomology in an exact sequence.

\begin{corollary}
\label{cor:Hypercohomology}
    Suppose $F_\bullet \to I$ is a resolution of a sheaf $I$ of length $n$ such that $H^i(X, F_j) = 0$ for all $j$ and all $1\le i \le n+1$. Then also $H^1(X, I) = 0$. If moreover $H^0(X, F_0) = 0$ then also $H^0(X, I) = 0$.\qed
\end{corollary}
To deal with flops we also need the following lemma which yields a comparison of cohomology groups for birational varieties.
\begin{lemma}
\label{lem:Comparison}
    Let $F$ be a locally free sheaf on a smooth variety $X$ and $Y\subset X$ of codimension $c>0$. Denote the complement by $U = X\backslash Y$. Then $H^i(X, F)\cong H^i(U, F|_U)$ for all $1\le i\le c-2$. 
\end{lemma}
\begin{proof}
    By \cite[Prop. 1.11]{GrothendieckLocal} the restriction morphism $H^i(X, F)\to H^i(U, F|_U)$ is an isomorphism if the local cohomology sheaves $\mathcal{H}^i_Y(F)= 0$ vanish for all $i\le c-1$. On the other hand, this vanishing is equivalent to $\depth_Y F = \min_{x\in Y} \depth F_x \ge c$ by \cite[Theorem 3.8]{GrothendieckLocal}. Thus, we compute the latter: $F_x = \oO_{X,x}^{\oplus n}$ has depth 
    \begin{equation*}
        \depth F_x \ge \depth \oO_{X,x} = \dim \oO_{X,x} \ge c
    \end{equation*}
    as $X$ is smooth and $x\in Y$ is a point of codimension at least $c$. 
\end{proof}

\begin{proof}[Proof of \Cref{prop:NotEffective}]
    For any $N\in \NN$ the ideal sheaf sequence for $Y$ yields an exact sequence
    \begin{equation*}
        0\to D^N\otimes I_Y \to D^N \to D^N|_Y\to 0
    \end{equation*}
    Thus, we only need to show that $H^1(X,D^N\otimes I_Y) = 0$ for all $N\gg 0$.
    In any case we have a resolution by locally free sheaves $F_{c-1}\to \ldots \to F_0 \to I_Y$ by \Cref{lem:Resolution} which when tensoring with $D^N$ yields a resolution of $D^N\otimes I_Y$. By  \Cref{cor:Hypercohomology} it suffices to show that $H^j(X, F_i\otimes D^N) = 0$ for all $j\le c$.\par
    
    Thus, in the case that $M$ is ample, Verbitsky's vanishing \Cref{thm:Verbitsky} gives the claim.\par
    
    Now assume that $M$ is movable and we are in the second case. Denote by $Z \subset X$ the exceptional locus of the birational map $f\colon X\dashrightarrow X'$ such that $M = f^*A$ for some ample $A\in \Pic X'$ and $X'$ a Hyperkähler variety as in \Cref{rem:MovableDivisorsAmple}. Then the codimension of $Z$ is greater than $c+2$ and by \Cref{lem:Comparison} we have $H^j(X, D^N\otimes F_i) = H^j(X\backslash Z , D^N\otimes F_i)$ for all $j\le c$. But as we may choose the $F_i$ to be the Koszul complex we may assume that the $F_i$ are sums of line bundles. Therefore they extend to vector bundles $F'_i$ on $X'$ as any line bundle extends over a codimension $\ge 2$ subset. Hence, $H^j(X\backslash Z , D^N\otimes F_i) = H^j(X' , {D'}^N\otimes F'_i) = 0$ for $N\gg 0$ by Verbitsky's theorem, where $D'$ corresponds to $D$ via the birational model and satisfies $\bbf{D'}{A} = \bbf{D}{M}<0$ as the Beauville-Bogomolov-Fukiji-form is preserved under passing to this birational model. 
\end{proof}

Thus, if $Y\subset X$ is a complete intersection, we have now shown that $\Eff(E)\subset \Eff(X)|_E$. The following finishes the proof of the main result \Cref{mainthm1}, showing that the inclusion is actually an equality
\begin{equation*}
    \Eff(E) = \Eff(X)|_E.
\end{equation*}

\begin{proof}[Proof of \Cref{mainthm1}]
    The Lefschetz-theorem for the Picard group (see \cite[IV.3.3]{HartshorneAmpleSubvarieties}) yields that $\Pic(X)\to \Pic(Y)$ is an isomorphism. Thus, by \Cref{prop:NotEffective} we only need to show that pseudo-effective divisors restrict to pseudo-effective divisors.
    
    By the Zariski-decomposition, see \Cref{thm:ZariskiDecomposition},  it suffices to show that movable divisors $M$ and exceptional divisors $E$ restrict to pseudo-effective ones on $Y$. 
    
    As in \Cref{rem:MovableDivisorsAmple} there exists a birational map $f\colon X\dashrightarrow X'$ to a Hyperkähler variety $X'$ and a nef divisor $A$ such that $f^*A = M$. Changing $M$ slightly we may assume that $A$ is ample. Thus, the base locus of $M$ is contained in the exceptional locus of $f$ and by assumption this has codimension $c+2$. Therefore $M|_A$ is effective.
    
    To prove effectivity for an exceptional divisor $E$ let $F_\bullet \to I_Y$ be the Koszul complex. Thus, $F_i = \oplus L_i$ for some ample line bundles $L_i$. By \Cref{cor:EffMovDual} there exists a divisor $N\in \Mov(X)$ such that $\bbf{N}{E} = 0$. Again, as in \Cref{rem:MovableDivisorsAmple} on a birational model $X\dashrightarrow X'$ the divisor $N$ is big and nef and $\bbf{N}{E-c_1(L_i)} < 0$. With the same method as in \Cref{prop:NotEffective} we get that $H^j(\oO(E)\otimes L_i^\vee) = 0$ for all $j\le c$ and thus by \Cref{cor:Hypercohomology} we have $H^0(X, E\otimes I_Y) = 0$. Therefore the short exact sequence
    \begin{equation*}
        0\to \oO(E)\otimes I_Y \to \oO(E) \to \oO_Y(E)\to 0
    \end{equation*}
    gives that $H^0(Y, \oO(E)) \neq 0$.
\end{proof}

\section{Effective cones of projectivized bundles}
\label{sec:EffConeExc}
In this section we will compute the effective cone of some projectivized vector bundles $E = \PP_S(\fF)$ on a K3 surface $S$. This will be done by embedding $E\hookrightarrow X$ into a Hyperkähler variety $X$. In most cases these will not be complete intersections but contained in the exceptional set of a birational morphism $X\to\overline{X}$. This leads us to generalize the approach of the last section to these subvarieties as well.\par
We start by giving some examples and elaborate on the construction thereof. 
\begin{construction}[{See overview article \cite{macrì2019lectures}}]
    Let $S$ be a K3 surface and $\sigma$ a stability condition, see \cite{macrì2019lectures} for details and definitions. Then, for a primitive $v\in \ZZ\oplus \Pic(S)\oplus \ZZ =\colon \Halg(S), \phi\in \RR$ and if $\sigma$ does not lie on a wall the moduli space every semi-stable object of Mukai vector $v$ is stable and
    \begin{equation*}
        X = M_\sigma(v, \phi) = \{E\in D^b(S)\,|\, E \;\text{is}\; \sigma\text{-stable, }v(E) = v\text{ and has phase }\phi\}
    \end{equation*} is a projective Hyperkähler variety with $\dim X = v^2+2$.
    The natural pairing
    \begin{align*}
        \Halg(S) \times \Halg(S) &\to \ZZ & (a,D,b),(a', D', b') &\mapsto D.D' - ab'-ba'
    \end{align*}
    is compatible with the Beauville-Bogomolov-Fujiki form in the following sense.
    If $v^2\ge 2$ there es an isomorphism
    \begin{equation*}
        \phi\colon v^\perp \cong \Pic(X)
    \end{equation*}
    which respects $\bbf{-}{-}$.\par
    A classical example, which is also the most important one for us, arises as follows: For a general ample polarization $H\in\Pic(S)$ the moduli space of Gieseker-stable sheaves, see \cite[Section 1.2, Section 4]{HuybrechtsLehnSheaves} for definitions and constructions,
    \begin{equation*}
        M_H(v) = \{E\in \textup{Coh}(S)\,|\, E \text{ is Gieseker-stable and }v(E) = v\}
    \end{equation*}
    is isomorphic to a Bridgeland moduli space.
\end{construction}

The benefit of dealing with the more general Bridgeland moduli spaces is the fact that the whole birational behaviour is well understood by the work of Bayer--Macr\`{i} \cite{bayer2014mmp}.
\begin{remark}
\label{ex:BayerMacri}
    In the setup above, the effective, movable and nef cone is determined as follows, see \cite[Section 12]{bayer2014mmp}: Choose an ample line bundle $A\in v^\perp$ in the orthogonal complement and denote by $\Pos(X) \subset \Pic(X)_\RR$ the component of the positive cone that has positive Beauville-Bogomolov-Fujiki intersection with $A$. Then the \emph{effective cone} $\Eff(X)$ is generated by 
    \begin{itemize}
        \item $\Pos(X)$,
        \item $\phi(s)$ for any class $s\in v^\perp$ with $\bbf{s}{s} = -2$ and $\bbf{s}{A}> 0$,
        \item $\phi\bigl(\bbf{v}{v}w - \bbf{v}{w}v\bigr)$ for classes $w\in \Halg(S)$ with $\bbf{w}{w} = 0$ and $\bbf{w}{v} = 1,2$. 
    \end{itemize}
    The \emph{movable cone} $\Mov(X)$ is dual to the effective cone and is consequently cut out in $\Pos(X)$ by the hyperplanes
    \begin{itemize}
        \item $\phi(s^\perp\cap v^\perp)$ for any class $s\in v^\perp$ with $\bbf{s}{s} = -2$ and $\bbf{s}{A}> 0$,
        \item $\phi(w^\perp\cap v^\perp)$ for classes $w\in \Halg(S)$ with $\bbf{w}{w} = 0$ and $\bbf{w}{v} = 1,2$. 
    \end{itemize}
    The \emph{nef cone} $\Nef(X)$ on the other hand is cut out in $\Pos(X)$ by
    \begin{itemize}
        \item $\phi(a^\perp\cap v^\perp)$ for all $a\in \Halg(S)$ with $\bbf{a}{a}\ge -2$ and $0\le \bbf{a}{v}\le 2$.
    \end{itemize}
    Moreover, the description in \textit{loc.cit.} shows that the nef cone is cut out from the movable cone by one hyperplane for each flopping contraction.
\end{remark}
We continue with two examples that carry exceptional subvarieties whose structure delivers some information on the K3 surface $S$ that we started with.
\begin{example}
    \label{ex:CotangentBundle}
    Let $X = S^{[2]}$ be the Hilbert scheme of length $2$ for a K3 surface $S$. This is a Hyperkähler fourfold isomorphic to $M(1,\oO_S,-1)$ and the Picard lattice is closely related to that of $S$, that is $\Pic(X) = \Pic(S)\oplus \frac{1}{2}E$, where $E$ is a prime divisor with $\bbf{E}{E} = -8$. The latter divisor $E$ arises as the exceptional divisor of the Hilbert-Chow morphism $S^{[2]}\to S^{(2)}$, which yields the isomorphism $E \cong \PP(\Omega_S)$ and the Hilbert-Chow morphism restricts to the projection $\PP(\Omega_S)\xrightarrow{p} S$. The corresponding morphism is induced by a big and nef line bundle $A^{[2]} = \phi(0,-A, 0)$ for $A\in \Amp(S)$. Moreover we have that the restrictions satisfy
    \begin{align*}
        A^{[2]}|_E &= 2\Tilde{A}\\
        E|_E &= -2L,
    \end{align*}
    where $\Tilde{A} = p^*A$ and $L = \oO_{\PP(\Omega)}(1)$.
\end{example}

We next turn our example to other $\mu$-stable bundles. It turns out that their projectivization can be embedded into Hyperkähler manifolds in many cases.
\begin{construction}[see {\cite[Example 8.1.7]{HuybrechtsLehnSheaves}}]
    \label{ex:MukaiBundles}
    Let $S$ be a K3 surface, $H\in \Pic X$ ample and $\fF$ a $\mu_H$-stable locally free sheaf  with Mukai vector $v = v(\fF) = (r, D, c)\in \Halg(S)$.
    We want to construct a closed embedding
    \begin{equation*}
        \PP_S(\fF) \to M_H(r,D,c-1).
    \end{equation*}
    For any point $s\in S$ and $z\in \PP(\fF(s))$ we define
    \begin{equation*}
        F_z = \ker (\fF\to \fF(s) \xrightarrow{z} k(s)).
    \end{equation*}
    Then $F_z\in M_H(r,D,c-1)$ is stable and 
    \begin{align*}
        \PP_S(\fF) &\to X  \qquad z \mapsto F_z
    \end{align*} 
    is indeed an embedding. \par 
    The moduli space $M_H(r,D,c-1)$ admits a birational contraction morphism 
    \begin{equation*}
    \label{eq:contraction}
        \textup{cont}\colon M_H(r,D,c-1) \to M^{\mu ss},
    \end{equation*}
    where $M^{\mu ss}$ is the Donaldson-Uhlenbeck compactification of the space of $\mu_H$-stable locally free sheaves, see \cite[Section 8.2]{HuybrechtsLehnSheaves} for details. An analysis of the contraction reveals that the contraction restricts and factorizes as $$\textup{cont}|_{\PP(\fF)}\colon \PP_S(\fF)\to S\subset M^{\mu ss},$$
    where the first map is the usual projection $p\colon \PP_S(\fF)\to S$.\par
    Assume now, that $(r,D, c-1)$ is primitive in $\Halg(S)$, e.g.\ $D$ primitive. By changing the polarization $H$ slightly, such that $H$ does not lie on a wall, we can assume that $M_H(r,D,c-1)$ is a Hyperkähler variety of dimension $2r+2+v^2$. The birational morphism to the Donaldson--Uhlenbeck compactification is induced by $\phi(0, rH, H.D) \in \Pic(X)$ for any ample $H\in \Pic(S)$. Moreover calculating the restrictions of line bundles we get that for  any vector $(r',D',c')\in (r,D,c-1)^\perp$ 
    \begin{align*}
        \phi(r',D',c')|_{\PP(\fF)} = -r'L +p^*D',
    \end{align*}
    where $L = \oO_{\PP(\fF)}(1)$.
\end{construction}

\begin{remark}
    To prove the Lefschetz-type theorems in the last section the codimension of the subvariety had to be strictly less than half of the dimension of the Hyperkähler variety to be able to apply  Verbitsky's vanishing theorem. Therefore we restrict in our further analysis to the case of $\mu_H$-stable bundles with $v(\fF)^2=-2$.
\end{remark}

One prominent example of stable bundles arises as Lazarsfeld--Mukai bundles which play a fundamental role in Brill-Noether theory, see \cite{LazarsfeldBrillNoether}. We recall the basic construction.
\begin{example}
    Let $S$ is a K3 surface of Picard rank $1$, that is $\Pic(S) = \ZZ H$, and $C\in |H|$ a smooth curve of genus $g$. Suppose furthermore that we are given a globally generated line bundle $A\in \Pic(C)$ of degree $d$ such that $A^\vee \otimes \oO_C(H)$ is globally generated as well. Then with $r = h^0(C,A)-1$ we get an exact sequence
    \begin{equation*}
        0 \to F_{C,A}\to \oO_X^{\oplus r+1} \to A \to 0
    \end{equation*}
    The \emph{Lazarsfeld--Mukai-bundle} is defined to be the dual $E_{C,A} = F_{C,A}^\vee$. This bundle has Mukai vector $v = (r+1, H, r-d+g)$. It turns out that this bundle is $\mu_H$-stable, see \cite[Prop.\ 9.3.3]{huybrechts2016lectures}.
    If we define $\rho(r,g,d) = g -(r+1)(g-d+r)$ then $v(E_{C,A})^2 = 2\rho(r,g,d)-2$. 
\end{example}

It turns out that another geometric construction of rigid bundles arises when considering the restriction of the cotangent bundle of projective space under some embedding:
\begin{example}
\label{ex:cotangentprojectivespaceonK3}
    Let $S$ be a K3 surface with $\Pic(S) = \ZZ H$ with ample, globally generated divisor $H$ of degree $2d = H.H$. Then, the kernel $M_H$ of the evaluation map $H^0(S, H)\otimes \oO_S \to H$ is $\mu_H$-stable. Moreover $v(M_H) = (d+1, -H, 1)$ satisfies $v(M_H)^2 = -2$ and for the morphism $S\to \PP^{N}$ induced by $H$ it holds that
    \begin{equation*}
        \Omega_{\PP^N}|_S = M_H \otimes H^\vee.
    \end{equation*}
\end{example}

\begin{remark}
\label{rem:AmplesOnExceptional}
    In both examples above, for any ample divisor $A\in \Pic(S)$ there exists a big and nef divisor $L_A$ which induces the morphism to the Uhlenbeck compactification and restricted to the exceptional set $E \to S$ is the pullback of a positive multiple of $A$, see \cite[Section 8.2 and Theorem 8.2.8]{HuybrechtsLehnSheaves}.
\end{remark}

To compute the effective cone of some of the exceptional subvarieties mentioned above we need to introduce the following notation.
\begin{notation}
Let $\mathcal{C}\subset N^1(X)$ be any subset. Then we define
\begin{equation*}
    \mathcal{C}_{E\ge 0} = \{D\in C\,|\, \bbf{D}{E} \ge 0\}.
\end{equation*}
\end{notation}
This notation will most often be used with the effective and movable cones $\mathcal{C} = \Eff(X), \Mov(X)$.

We are now ready to prove the second main theorem \Cref{mainthm2} in a more general version, which is applicable to rank $2$ bundles as above.
\begin{theorem}
\label{thm:Eff4Fold}
    Let $S$ be a K3 surface and $H\in \Pic(S)$ ample. 
    \begin{itemize}
        \item Let $\fF$ be a $\mu_H$-stable rigid rank $2$ vector bundle with $v(\fF) = (2,D,c)$ such that $v' = (2,D,c-1)$ is primitive. Then choose an $H'\in \Pic(S)$ nearby $H$ such that $H'$ does not lie on a wall. Set
        \begin{equation*}
            X = M_{H'}(2,D,c-1)\supset \PP(\fF) = E,
        \end{equation*}
        or
        \item $X = S^{[2]}$ the Hilbert scheme and $E = \PP(\Omega_S)$.
    \end{itemize} 
    Suppose that $X$ does not admit any flopping contractions. Then the restriction morphism induces a bijection
    \begin{equation*}
        \Eff(E) \cong \Eff(X)_{E\ge 0}.
    \end{equation*}
\end{theorem}

\begin{proof}
    The restriction morphism on the Picard group induces a bijection
    \begin{equation*}
        \Pic(X)_\QQ \to \Pic(E)_\QQ.
    \end{equation*}
    We first show that any divisor $D\in \Pic(X)$ with $\bbf{D}{E} < 0$ is not effective. But to show this it suffices to show that all divisors with $\bbf{D}{E} = 0$ are not big which is immediate as these divisors are pull backs under the map $E = \PP(\fF) \to S$.

    Therefore by \Cref{prop:NotEffective} we only need to show that the restriction of any divisor in $\Eff(X)_{E\ge 0}$ is again pseudo-effective. But the latter cone is generated by exceptional divisors $E'\neq E$, the movable cone $\Mov(X)$ and $\Eff(X)_{E=0}$. The claim for the first two cones is immediate so we are left with effective divisors $D\in \Pic(X)$ such that $\bbf{D}{E} = 0$. Then $D|_E$ is a pullback of a divisor $D'\in \Pic(S)$ on $S$ and we want to show that $D'$ is effective. Suppose this is not the case, then there exists an ample $A\in \Pic(S)$ such that $A.D' < 0$ but on the other hand the big and nef divisor $L_A$ from \Cref{rem:AmplesOnExceptional} satisfies $\bbf{L_A}{D} = A.D' < 0$ and therefore $D$ cannot be effective.
\end{proof}

The strategy in the higher dimensional cases will have to be altered: The bundle $\PP_S(\fF)$ is of higher codimension in the Hyperkähler variety. As this space is also contracted it will then be induced by a flopping contraction, in contrast to the divisorial contraction before in the rank $2$ case. We start with a subsection that analyses the geometry of these contractions.
\subsubsection*{Exceptional divisors are effective on $\PP(\fF)$}
The main tool we use are symplectic varieties in the sense of \cite{KaledinSymplecticOverview}. The main theorem regarding these varieties is the following.
\begin{theorem}[{Kaledin \cite[Theorem 2.3 and Theorem 2.5]{kaledinSymplecticPoisson}}]
    \label{thm:kaledinsymplectic}
    Let $X$ be a symplectic variety. Then there exists a stratification by locally closed subschemes $X_i$ that are symplectic and smooth. Moreover the normalizations of their closures are symplectic varieties as well.
\end{theorem}
\begin{remark}
    From \cite[Proof of Proposition 3.1]{kaledinSymplecticPoisson} the stratification is the smooth-singular stratification.
\end{remark}
The following sheds some light on this stratification in the case that there is a crepant resolution by a Hyperkähler variety.
\begin{lemma}
\label{lem:SingularLocusContraction}
Let $f\colon X\to \overline{X}$ be a birational projective morphism of a projective Hyperkähler variety of dimension $2n$ to a normal variety. Suppose that $f$ contracts only  a prime divisor $E$. Then $f(E)= \textup{Sing}(\overline{X})$.
\end{lemma}
\begin{proof}
We know that $\dim f(E) = 2n-2$ as the map $f$ is semi-small, see \cite[Lemma 2.11]{kaledinSymplecticPoisson}. Suppose $f(E)\not\subset \textup{Sing}(\overline{X})$. Denote by $X' \subset X$ the preimage of $\overline{X}^{sm}$. Then $X' \to \overline{X}^{sm}$ is birational, but contracts a codimension one subset. On the other hand $K_{\overline{X}^{sm}} = 0$ as the space $\overline{X}$ is normal and $f(E)$ has codimension $2$. Thus, this is a crepant morphism between smooth quasi-projective varieties and therefore must be an isomorphism: The short exact sequence 
\begin{equation*}
    f^*\Omega_{\overline{X}^{sm}} \to \Omega_{X'} \to \Omega_f \to 0
\end{equation*}
induces a map on the determinants of the first two locally free sheaves. But both determinants are isomorphic, thus the map on determinants is an isomorphism, as it is non-zero at the locus over which $f$ is an isomorphism. Therefore $f^*\Omega_{\overline{X}} \to \Omega_X$ is an isomorphism as well and it follows that $\Omega_f = 0$. However, this contradicts that the map contracts fibers, as the base change of $\Omega_f$ to a non trivial fiber via a map $\{x\}\to \overline{X}^{sm}$ is non trivial.
\end{proof}

We start with remarking the following two facts on normalizations.
\begin{lemma}
    Let $X\to Y$ be a finite birational morphism between projective varieties. Then there is a factorization of the normalization morphism $\Tilde{Y}\to Y$ as $\Tilde{Y}\to X\to Y$.
\end{lemma}
\begin{proof}
    As $X\to Y$ is dominant, we get a finite birational map $\Tilde{X}\to \Tilde{Y}$ between normal varieties. therefore by Zariski's main theorem this is an isomorphism.
\end{proof}
\begin{lemma}
\label{lem:normalization}
Let $X\to Y$ be a morphism from a normal irreducible projective variety to a projective variety, such that the smooth locus $U\subset Y$ satisfies $f(X)\cap U\neq \varnothing$. Then there is a factorization $X\to \Tilde{Y}\to Y$ through the normalization $\Tilde{Y}$.
\end{lemma}
\begin{proof}
    Denote by $Z = f(X)\subset Y$ the image, which is closed. Then by assumption $U\cap Z$ is dense and open. Thus, the normalization $\Tilde{Y}\to Y$ induces a birational map $Z'\to Z$, where $Z'\subset \Tilde{Y}$ is the closure of the preimage of $U\cap Z$. The map is also finite as $\Tilde{Y}\to Y$ is finite. By the universal property of the normalization there is a factorization $\Tilde{Z}\to Z'\to Z$, where $\Tilde{Z}$ is the normalization of $Z$. On the other hand by the universal property, as $X\to Z$ is dominant, we get a factorization $X\to \Tilde{Z}\to Z$. Therefore, there is a map $X\to Z'\to \Tilde{Y}$.
\end{proof}
We will analyze the behaviour of the fibers of $\PP(\fF)\to S$, which are projective spaces, under birational maps of Hyperkähler varieties to symplectic varieties. However, the symplectic structure puts restrictions on these maps. The following is an simple consequence of Kaledin's stratifictaion result. However, in the smooth case the following would also follow from a direct argument or \cite[Lemma 1.1]{voisinRemarksCoisotropic}.
\begin{lemma}
\label{lem:MapsToSymplectic}
    There are no generically finite maps $f\colon X\to Y$ from a rational variety $X$ of dimension $k$ to a symplectic variety $Y$ of dimension $< 2k$.
\end{lemma}
\begin{proof}
    Replacing $X$ with a resolution of singularities we may assume that $X$ is smooth. In the following we want to use that there are no non-zero holomorphic forms on any smooth rational variety. Now suppose that $f$ is generically finite.\par
    By \Cref{thm:kaledinsymplectic} there is a stratification $Y_i$ of $Y$ into locally closed strata such that all $Y_i$ are smooth and symplectic and moreover the normalization of the closures are symplectic varieties. Therefore the image is contained in the closure $f(X)\subset \overline{Y_i}$ and $f(X)\cap Y_i\neq \varnothing$ for one $i$. Hence, $X_i = f^{-1}(Y_i)\subset X$ is open and non empty as well and $X_i \to Y_i$ is generically finite. By \Cref{lem:normalization} we may replace $Y$ with the normalization of $\overline{Y_i}$ and assume that there is a generically finite morphism $X\to Y$, where $Y$ is a symplectic variety and $f(X)\cap Y^{sm} \neq \varnothing$.\par
    Let $r\colon Z\to Y$ be a resolution of singularities. Then there is a holomorphic two form $\omega_Z\in H^0(Z, \Omega_Z^2)$, which is non-degenerate on $r^{-1}(Y^{sm})$. On the other hand, there is a rational map $X\dashrightarrow Z$, as $f(X)\cap Y^{sm}\neq \varnothing$.
    In the following let $\Tilde{X}\to X$ be a smooth resolution of $f\colon X\dashrightarrow Z$. This forces the pullback $f^*\omega_Z$ on $\Tilde{X}$ to be non-trivial, as the map $\Tilde{X}\to Z$ is generically \'etale and thus the \'etale locus meets the locus where $\omega_Z$ is non-degenerate: Let $x\in \Tilde{X}$ be such a point. Then there is the exact sequence
    \begin{equation*}
        T_{\Tilde{X},x} \to T_{Z, f(x)} \xrightarrow{\omega_Z} T_{Z, f(x)}^* \to T_{\Tilde{X},x}^*
    \end{equation*}
    where the first map is injective, the middle one is bijective and the last one is surjective. But as $\dim  T_{\Tilde{X},x} = k$ and $\dim  T_{Z,f(x)} < 2k$ this forces the composition - and hence $f^*\omega_Z$ - to be non-zero. On the other hand the pullback of $\omega_Z$ is zero as $\Tilde{X}$ is rational, a contradiction.
\end{proof}

\begin{proposition}
\label{prop:FlopNotContainedInExceptional}
Let $X\to \overline{X}$ be a birational morphism of a $2n$ dimensional Hyperkähler variety $X$ to a normal projective variety $\overline{X}$ which contracts only an irreducible divisor $E$. Furthermore let $\PP(\stbdl)\subset X$ be the projectivization of a vector bundle of rank $n$ over a K3 surface $S$. Then either
\begin{itemize}
    \item every fiber of $p\colon \PP(\stbdl)\to S$ gets contracted to a point by $X\to \overline{X}$, or
    \item $\PP(\stbdl)\not\subset E$.
\end{itemize} 
\end{proposition}
\begin{remark}
Suppose that $\PP(\stbdl)\subset E$. The strategy is to first show that $X\to\overline{X}$ contracts $\PP(\stbdl)$ to an at most $n-1-$ dimensional subvariety, which is the image of a projective space $\PP^{n-1}$. Afterwards, we show that this contradicts that $X\to \overline{X}$ is a crepant resolution of a symplectic variety if the first condition is not met.
\end{remark}
\begin{proof}
Suppose we have $\PP(\stbdl)\subset E$ and no fiber of $p\colon \PP(\stbdl)\to S$ gets contracted by the birational morphism $X\to\overline{X}$. Furthermore denote $\overline{E} = f(E)$. 
We claim that $f(\PP(\stbdl)) = f(F)$ for all fibers $F\cong \PP^{n-1}$ of $\PP(\stbdl)\to S$.\par 
From Kaledin's theorem there is a stratification $X_i\subset \overline{X}$ of locally closed subsets that are smooth and symplectic. As the smooth part of $\overline{X}$ is precisely $\overline{X}\backslash\overline{E}$, we get that this is one of the strata. Therefore the other strata are all contained in $\overline{E}$ and cover it. Now, let $X_i$ be a stratum, such that $\overline{E}\cap X_i\neq \varnothing$ and $\overline{E}\subset \overline{X_i}$, where $\overline{X_i}\subset \overline{X}$ is the closure. It is necessarily of dimension $\le 2n-2$. By the previous lemma on the normalization $\PP(\stbdl)\to \overline{X_i}$ factors through the normalization $\Tilde{X_i}$ of $\overline{X_i}$. \par
At first observe that the fibers $\PP^{n-1}$ cannot get contracted as any map from  $\PP^{n-1}$ is either finite or constant and by assumption the latter does not happen.
On the other hand, for any rational curve $R\subset S$, the set $\PP(\stbdl|_{{R}})$ is a rational variety of dimension $n$.
But then it maps to the symplectic variety $\Tilde{X_i}$ of dimension $\le 2n-2$, thus it cannot be generically finite by \Cref{lem:MapsToSymplectic}.
By upper semi-continuity we then have $\dim f^{-1}(f(x)) \ge 1$ for all $x\in \PP(\stbdl|_{R})$. Therefore, $f(\PP(\stbdl|_{R})) = f(F)$ for $F$ a fiber over $R$.\par
There is a sequence of rational curves $(R_i)_{i\in \NN}$ such that $R_i\cap R_j\neq \varnothing$ for all $i,j\in \NN$: From \cite{chen2019curves} there are infinitely many rational curves on any K3 surface and by Chen--Gounelas \cite[Equation 3.15]{chenGounelasMaximalModuli} the statement holds when $R_i^2$ is bounded. On the other hand if this number is unbounded, we can pick rational curves $R_i$ with positive self intersection. These satisfy $R_i.R_j > 0$, giving the claim. \par

Let $x_i\in \PP(\stbdl|_{R_i})$ and let $C_i\subset f^{-1}(f(x_i))\cap \PP(\stbdl|_{R_i})$ be a curve. This curve satisfies $p(C_i) = R_i$, as otherwise $C_i$ would be contained in a fiber of $p\colon \PP(\stbdl)\to S$, which gets contracted by $f$. This, however, would force the whole fiber $\PP^{n-1} = F = p^{-1}(p(x))$ to get contracted, which does not happen by assumption. \par
Thus, for any $j$ there is a point $x_j\in C_i$ such that $p(x_j) \in R_j$ as $R_i\cap R_j \neq \varnothing$ for all $i,j$. With the same argument we find curves $C_j\subset f^{-1}(f(x_j))$ such that $p(C_j) = R_j$ and therefore all these curves $C_j$ have to be pairwise distinct. On the other hand, by construction $f(x_i) = f(x_j)$ and therefore we conclude $\dim f^{-1}(f(x_j)) \ge 2$ for all $x_j\in \PP(\stbdl|_{R_j})$ as $\bigcup_j C_j\subset f^{-1}(f(x_i))$. \par

But then, the space of such points $x$ is dense in $\PP(\stbdl)$, and thus $\dim f^{-1}(f(x)) \ge 2$ for all $x\in \PP(\stbdl)$ by upper semi-continuity. 
Thus, $f(\PP(\stbdl)) = f(F)$, as the image is irreducible, contains the latter and both spaces have the same dimension. \par

The symplectic form $\sigma_X$ restricts to a non trivial $2$-form on $\PP(\stbdl)$, as $2\dim \PP(\stbdl)> \dim X = 2n$. Hence, it is the pullback of the two form on the K3 surface $S$ as $H^2(\PP(\stbdl), \oO_\PP) \cong H^2(S, \oO_S) = \CC$.\par
Denote by $Z\subset f(F)$ the open subset of the smooth part of $f(F)$ such that $g\colon W = \PP(\stbdl)\cap f^{-1}(Z)\to Z$ is smooth. By Kaledin \cite[Lemma 2.9]{kaledinSymplecticPoisson} we can shrink $Z$ to obtain a $2$-form $\sigma_Z$ on $Z$, such that the $2$-form $\sigma_X$ of $X$ satisfies $\sigma_X|_W = g^*\sigma_Z$. As $W\subset \PP(\stbdl)$ is open, $\sigma_X|_W$ and $\sigma_Z$ cannot be trivial. On the other, denote by $F'$ a non empty intersection of a fiber $F$ of $p\colon \PP(\stbdl)\to S$, i.e. $F' = F\cap W$. Then we have the composition 
\begin{equation*}
    F' \to W \to Z
\end{equation*}
which is finite. Therefore the pullback of $\sigma_Z$ to $F'$ cannot be trivial. But the commutative diagram
\begin{equation*}
\begin{tikzcd}[row sep=huge]
\PP^{n-1} \arrow[r, hook] &X\arrow[r, two heads] &\overline{X}\\
F'\arrow[r, hook]\arrow[u, hook]&W\arrow[u, hook] \arrow[r, two heads]&  Z\arrow[u, hook]
\end{tikzcd}
\end{equation*}
shows that $\sigma_X|_{F'} = \sigma_X|_{\PP^{n-1}}|_{F'} = 0$, a contradiction.
\end{proof}

Arguing in the same way one proves:
\begin{theorem}
    Let $X$ be a $2n$-dimensional hyperkähler variety and $X\to \overline{X}$ a birational morphism that contracts only a divisor $E$. Then for any $\PP^n\subset X$ it holds that either
    \begin{itemize}
        \item $\PP^n$ gets contracted by $X\to \overline{X}$, or
        \item $\PP^n\not\subset E$.
    \end{itemize}
    
\end{theorem}
\begin{question*}
    Let $X\to \overline{X}$ be a birational morphism that contracts a divisor $E$ and $F\subset X$ the exceptional locus of a flopping contraction. Is $F\not\subset E$?
\end{question*}

Before stating the general result we remark the following, which gives some bounds on the effective cone.
\begin{remark}
\label{rem:ClassCuttingAmple}
	The \Cref{ex:MukaiBundles} deals with the embedding $\PP(\fF) \subset M_H(v)$. It turns out that there is a class $B \in \Halg(S)$ such that any class in $\Pic(X)_{B = 0}|_E \subset \Pic(E)$ is just the pull back of a line bundle in $\Pic(S)$ under the map $E= \PP(\fF) \to S$. Thus, any such class is not big.
\end{remark}
The following now proves \Cref{mainthm3}. Let $B$ be as in the remark above.
\begin{theorem}
    \label{cor:MukaiEffHigherDim}
    Let $S$ be a K3 surface with $\Pic(S) = \ZZ H$ and $\fF$ be an $\mu_H$-stable vector bundle with Mukai vector $v = (r,D, c)\in \Halg(S)$ and $r>1$. Assume that $v^2 = -2$ and $(r,D,c-1)$ primitive. Denote by $F = \PP(\fF)$ the projectivization and by $X = M_H(r,D,c-1)$ the moduli space of stable sheaves. Suppose further that $\Nef(X)_{B\ge 0} = \Mov(X)_{B\ge 0}$. Then the induced morphism by the inclusion $F\hookrightarrow X$ 
    \begin{equation*}
        \Eff(X)_{B \ge 0} \to \Eff(F)
    \end{equation*}
    is an isomorphism.
\end{theorem}

\begin{proof}
    As the Picard rank of $S$ is $1$ the moduli space $X$ is a Hyperkähler manifold and we can mimic the proof of \Cref{thm:Eff4Fold}. By the remark above any class $D \in \Eff(X)_{B< 0}$ restricts to a non-effective divisor $D|_F$. As the Picard rank of $X$ is $2$, there is exactly one extremal ray $R$ of $\Eff(X)$ with $\bbf{R}{B}>0$. Therefore by assumption we can assume that there is a nef divisor $N$ which satisfies $\bbf{N}{R} = 0$ and $\bbf{N}{D}< 0$ for any divisor nearby $R$ but outside the effective cone. 
    Arguing as in the proof of \Cref{thm:Eff4Fold} and by \Cref{prop:NotEffective} we know that any non-effective divisor $D$ satisfies $H^0(\PP(\fF), D|_{\PP(\fF)}) = 0$ as well.\par 
    Therefore, we only need to show that the restriction of any effective exceptional divisor $E$ with $\bbf{B}{E} \ge 0$ is effective: Any such divisor is birationally contractible and the assumption on the nef cone yields that $E$ is actually contractible in $X$ itself. Then, by \Cref{prop:FlopNotContainedInExceptional} we have that $F\not\subset E$ and thus, $E|_F$ is effective.
\end{proof}
In the same way one can show that $\Eff(X)|_{\PP(\fF)} \subset \Eff(\PP(\fF))$ in any Picard rank. However, the assumption on the Picard rank was crucial, as we could not assure, that for any non-effective divisor $D$ with $\bbf{D}{B} > 0$ we find an \emph{ample} divisor which is negative on $D$.
\begin{question*}
    Does \Cref{cor:MukaiEffHigherDim} also hold for $\rk \Pic(S) > 1$?
\end{question*}


\section{Computation of effective cones of rigid stable bundles and the cotangent bundle}
\label{sec:Computations}
In this section we will provide some numerics for some projectivized bundles on K3 surfaces with $\Pic(S) = \ZZ H$. This builds on the characterisation of the effective cone, see \Cref{ex:BayerMacri}.
\subsection*{Cotangent bundle}
We start with discussing the cotangent bundle $\Omega_S$ for a K3 surface $S$ in detail. Suppose $\Pic S = \ZZ H$. In \cite{ottem2020remarks} Gounelas--Ottem computed the nef (resp. effective) threshold for the cotangent bundle $\PP(\Omega)\xrightarrow{p} S$ in some degrees. We have that the Hilbert-Chow exceptional divisor $E\subset S^{[2]} = X$ is isomorphic to $\PP(\Omega_S)$ and the contraction is induced by a line bundle $H^{[2]}\in \Pic(S^{[2]})$ which restricts to $2p^*H$, as in \Cref{ex:CotangentBundle}. Furthermore the adjunction formula yields $\oO_X(E)|_E = \oO_\PP(-2) = -2L$. 

The four-fold case is special, as we have a complete description of flops occurring on them, i.e. they are all Mukai flops along a $\PP^2\subset X$. This allows us to gain more insight in the ample cone as follows: The lemma below also occurred in \cite{ottem2020remarks}.
\begin{lemma}
    Let $X$ be a Hyperkähler fourfold with Picard rank $2$ that admits a flopping contraction $X\to \overline{X}$ of a $\PP^2$. Then any contractible exceptional divisor $E\subset X$ meets $\PP^2$ in a curve.
\end{lemma}
\begin{proof}
    Let $A\in \Pic X$ be the big and nef divisor inducing the flopping contraction $X\to \overline{X}$ which contracts $\PP^2$ to a point. Suppose the intersection $\PP^2\cap E = \emptyset$. Then $E.C = 0 = A.C$ for any line $C\subset \PP^2$. But as  by the assumption on the Picard rank we have that then also $H.C = 0$, a contradiction. On the other hand $\PP^2\subset E$ also leads to a contradiction, as the divisor $E$ gets contracted to a symplectic variety of dimension $2$, which again would yield a non-trivial holomorphic two form on $\PP^2$.
\end{proof}

This allows us to extend the computation in \cite[Theorem 3.8]{ottem2020remarks}: With the method of \Cref{sec:EffConeExc} and the Lemma above we can for any degree compute at least one of the effective or the nef cone, depending on whether $S^{[2]}$ admits a flop.
\begin{corollary}
    Let $S$ be a K3 surface with $\Pic(S)= \ZZ H$ of degree $H^2 = 2d$. 
    Suppose that the equation $x^2-4dy^2 = 1$ 
    \begin{itemize}
        \item has a minimal solution $(x_0,y_0)$. Then $$\Nef(\PP(\Omega_S)) = \langle \Tilde{H}+d\frac{y_0}{x_0}L,\Tilde{H} \rangle.$$
        \item has no solution. If
        \begin{itemize}
            \item $d$ is a perfect square $d= t^2$, then $$\Eff(\PP(\Omega_S)) = \Nef(\PP(\Omega_S)) = \langle \Tilde{H}+\frac{t}{2}L,\Tilde{H} \rangle$$ 
            \item $d$ is not a perfect square, the equation $x^2-dy^2 = 1$ has a minimal solution $(x_0,y_0)$ and
            $$\Eff(\PP(\Omega_S)) = \langle \Tilde{H}+\frac{x_0}{2y_0}L,\Tilde{H} \rangle.$$
        \end{itemize}
    \end{itemize}
\end{corollary}
\begin{proof}
    This directly follows from the computations of the effective and ample cone of $S^{[2]}$ as in \cite[Section 13]{bayer2014mmp}
\end{proof}
\begin{example}
    Writing the non-trivial extremal ray of the nef (resp. effective) cone as $L+ \alpha_n \Tilde{H}$ (resp. $L+\alpha_e \Tilde{H}$) we get the following numbers:
    \renewcommand{\arraystretch}{1.5}
    \begin{table}[H]
    \begin{tabular}{l|lllllllllllllllllll}
    $H^2$& 2 & 4 & 6 & 8 & 10 & 12 & 14 & 16 & 18 & 20\\ \hline
    $\alpha_e$&? & $\frac{4}{3}$ & $1$ & $1$ & ?  & $\frac{4}{5}$ & $\frac{3}{4}$ & $\frac{2}{3}$ & $\frac{2}{3}$ & $\frac{12}{19}$\\ \hline
    $\alpha_n$& 3 &? &? & 1 & 1 &? &? &? & $\frac{2}{3}$ & ?
    \end{tabular}
    \end{table}
    \begin{table}[H]
    \begin{tabular}{l|lllllllllllllllllll}
    $H^2$& 22 & 24 & 26 & 28 & 30 & 32 & 34 & 36 & 38 \\ \hline
    $\alpha_e$& ?  & $\frac{4}{7}$ & $\frac{360}{649}$ & $\frac{8}{15}$ & $\frac{1}{2}$ & $\frac{1}{2}$ & $\frac{16}{33}$ & $\frac{8}{17}$ & ?
    \\ \hline
    $\alpha_n$& $\frac{7}{11}$ & ? &? &? &? &$\frac{1}{2}$ &? &?  &$\frac{9}{19}$
    \end{tabular}
\end{table}
\end{example}

\subsection*{Lazarsfeld--Mukai bundles}
Let $S$ be a K3 surface with $\Pic(S) = \ZZ H$. In this subsection we will compute the effective cones of some rigid $\mu$-stable vector bundles on $S$. The prime example to keep in mind are Lazarsfeld--Mukai bundles $\fF$ with $\rho(r,g,d) = 0$ or equivalently $\bbf{v(\fF)}{v(\fF)} = -2$, see \Cref{ex:MukaiBundles}. However, not all rigid stable bundles arise in this way.
\begin{remark}
    For any prescribed Mukai-vector of the form $(r,H,c)$ with $r>0$ and $v^2 = -2$ there exists a $\mu_H$-stable bundle: There exist (Gieseker)-stable bundles with Mukai vector $v$, see e.g. \cite[Ch. 10, Thm. 2.7, Rem. 3.2]{huybrechts2016lectures}. Suppose there is $F\subset \fF$ with $0<\rk F < \rk \fF$ and $\deg F \,\rk \fF = \deg \fF \,\rk F$. As the Picard rank of $S$ is $1$, we get $\deg F = k(H.H)$ and thus
    \begin{equation*}
        k \,\rk \fF = \rk F,
    \end{equation*}
    a contradiction.
\end{remark}
\subsection*{Lazarsfeld--Mukai bundles of rank $2$}
We first deal with the case of bundles of rank $2$. Let $S$ be a K3 surface with $\Pic(S) = \ZZ H$. The assumption that $\fF$ is rigid imposes some conditions on the degree $H.H = 2d$, namely either
\begin{equation*}
    H.H = \begin{cases} 8k+2\\
    8k+6
    \end{cases}
\end{equation*}
for some $k\in \NN$.
\subsection*{The case $H.H = 8k+2$}
In this case the Mukai vector satisfies $v(\fF) = (2, H, 2k+1)$.
As in \Cref{ex:MukaiBundles} we have that $E = \PP(\fF)$ is the exceptional divisor in the Hyperkähler fourfold $X = M_H(2, H, 2k)$.
With the work of Bayer--Macr\`{i}, see \Cref{ex:BayerMacri}, we see that $X$ does not have a flopping contraction and the following description for the effective cone holds.
\begin{theorem}
    Let $S$ be a K3 surface with $\Pic(S) = \ZZ H$ of degree $8k+2$ and $\fF\in M(2,H,2k+1)$ a $\mu$-stable vector bundle. Then if $4k+1$ 
    \begin{itemize}
        \item is a perfect square, i.e. $4k+1 = t^2$, then 
            $$\Eff(\PP(\fF)) = \Nef(\PP(\fF)) = \langle \Tilde{H}, \frac{1-t}{2t}\Tilde{H} + L\rangle$$
        \item is not a perfect square, then 
        $$(8k+2)x^2 - (4k+1)xy -2ky^2 = -2$$
        has a solution with $\frac{y}{x} < 0$ maximal. Then
        $$\Eff(\PP(\fF)) = \langle \Tilde{H}, \frac{x}{y}\Tilde{H} + L\rangle$$
    \end{itemize}
\end{theorem}
\begin{proof}
Let $v=(2,H,2k)$ be the Mukai vector for the moduli space $X = M_H(v)\supset \PP(\fF)$.
To observe whether there exists a flop, we need to check if there is an element $s \in \Halg(S)$ with $\bbf{s}{s} = -2$ and $\bbf{s}{v} = 1$.   But for $s = (x,yH,z)$ the second equation reads
\begin{equation*}
    (8k+2)y-2z-2kx = 1,
\end{equation*}
which is not possible for integral values $x,y,z\in \ZZ$. Thus, $\Mov(X) = \Nef(X)$ and by the prior work, $\Eff(\PP(\fF)) = \Eff(X)_{E\ge 0}|_E$. Thus, we are interested in the second extremal ray only. If there is an element $D\in \Pic(X)$ with $\bbf{D}{D} = 0$ then there exists a Lagrangian fibration. But this is the case if and only if $4k+1$ is a square. 

If this is not the case the effective cone is determined by the following equation by \cite[Theorem 3.16]{debarre2018hyperk}
\begin{equation*}
    (8k+2)x^2 - (4k+1)xy -2ky^2 = -2
\end{equation*}
with $\frac{y}{x}< 0$ maximal. Then the effective cone of $\PP(\fF)$ is given by 
\begin{equation*}
    \Eff(\PP(\fF)) = \langle \Tilde{H}, L + \frac{x}{y}\Tilde{H}\rangle.
\end{equation*}
\end{proof}
This gives the following examples in low degrees
\renewcommand{\arraystretch}{1.5}
 \begin{table}[H]
    \begin{tabular}{l|cccccccccc}
    $H^2$& 2 & 10 & 18 & 26 & 34 & 42 & 50 & 58 & 66 & 74\\ \hline
    $\alpha_e$&$0$ &$-\frac{1}{3}$  & $-\frac{1}{3}$ & $-\frac{4}{11}$ & $-\frac{25}{66}$ & $-\frac{2}{5}$ & $-\frac{2}{5}$ & $-\frac{11}{27}$ & $-\frac{19}{46}$ & $-\frac{61}{146}$\\
    $\alpha_n$&$0$&?&$-\frac{1}{3}$&?&?&?&$-\frac{2}{5}$&?&?&?
    \end{tabular}
\end{table}
\subsection*{The case $H.H = 8k+6$}
In this case the Mukai vector satisfies $v(\fF) = (2, H, 2k+2)$.
As in \Cref{ex:MukaiBundles} we have that $E = \PP(\fF)$ is the exceptional divisor in the Hyperkähler fourfold $X = M_H(2, H, 2k+1)$. However, in this case we get flops in $X$ in certain degrees. For $s=(x,yH,z)$ these are governed by the equations 
\begin{align}
\label{eq:nefcone}
\begin{split}
    y^2(8k+6)-2xz &= -2\\
    y(8k+6)-(2k+1)x-2z &= 1,
\end{split}
\end{align}
whereas the effective cone is governed by
\begin{equation}
\label{eq:effectivecone}
    (8k+6)x^2 - (4k+3)xy -2(k+1)y^2 = -2.
\end{equation}

\begin{theorem}
    Let $S$ be a K3 surface with $\Pic(S) = \ZZ H$ of degree $8k+6$ and $\fF\in M(2,H,2k+2)$ a $\mu$-stable vector bundle. Then if 
    \begin{itemize}
        \item \Cref{eq:nefcone} has solutions, then
            \begin{equation*}
                \Nef(\PP(\fF)) = \langle \Tilde{H}, L + \alpha_n\Tilde{H}\rangle,
            \end{equation*}
            where $\alpha_n = \frac{x(k+1)-z}{y(8k+6)-(4k+3)x}$ and $(x,y,z)$ solutions of \Cref{eq:nefcone} such that $\alpha_n<0$ is maximal.
        \item \Cref{eq:nefcone} has no solutions, then
            \begin{equation*} 
                \Eff(\PP(\fF)) = \langle \Tilde{H}, \frac{x}{y}\Tilde{H} + L\rangle,
            \end{equation*}
            where $x,y$ are solutions to \Cref{eq:effectivecone} with $\frac{y}{x} < 0$ maximal.
        
    \end{itemize}
\end{theorem}
\begin{proof}
    Follows analogously by the Bayer--Macr\`{i} formulas for the effective and nef cone, see \Cref{ex:BayerMacri}.
\end{proof}
This gives the following examples in low degrees
\renewcommand{\arraystretch}{1.5}
 \begin{table}[H]
    \begin{tabular}{l|cccccccccccc}
    $H^2$& 6 & 14 & 22 & 30 & 38 & 46 & 54 & 62 & 70 & 78 \\ \hline
    $\alpha_e$&$-\frac{1}{4}$ &$-\frac{5}{16}$  & ? & $-\frac{3}{8}$ & ? & $-\frac{19}{48}$ &  $-\frac{21}{52}$ & ? & $-\frac{5}{12}$ & $-\frac{21}{50}$\\ \hline
    $\alpha_n$& ? &? & $-\frac{7}{5}$ & ? & $-\frac{161}{418}$ & ? & ?  & $-\frac{25}{62}$ & ? & ? 
    \end{tabular}
\end{table}
\subsection*{Lazarsfeld--Mukai-bundles of higher rank}
Let $S$ be a K3 surface of Picard rank one and $\fF$ a $\mu$-stable vector bundle of Mukai vector $v = v(\fF) = (r,H, c)$ with $v^2 = -2$. Then as above $\PP(\fF)\subset M_H(r,H,c-1)$. Moreover any two such bundles with the same Mukai vector are isomorphic by stability. 
We will give a few examples for which \Cref{cor:MukaiEffHigherDim} is applicable. 
\begin{example}
    The following tables give examples for rigid $\mu_H$-stable vector bundles with the numerics given in the form $(r,g,d)$ as for Lazarsfeld--Mukai bundles, i.e. $v(\fF) = (r+1,H,r-d+g)$ and $H.H = 2g-2$. Each row shows a wall of the movable cone of the Hyperkähler variety $M_H(r+1,H,r-d+g-1)$. The labels of the columns are as follows
    \begin{itemize}[itemsep=5pt, topsep=5pt]
        \item $(r,g,d)$ gives the numerics of the vector bundle,
        \item $\alpha_e$ gives the non-trivial ray of $\Eff(\PP(\fF)) = \langle H,L-\alpha_eH \rangle$,
        \item \emph{Type} gives the type of the wall, i.e. either a Hilbert--Chow, Lie--Gieseker--Uhlenbeck, Brill--Noether, Lagrangian or a flopping contraction,
        \item \emph{Vector} gives the vector inducing the contraction in the sense of \cite[Theorem 5.7]{bayer2014mmp},
        \item \emph{Movable} shows a movable divisor which (birationally) induces the contraction
        \item \emph{Contracted} shows the contracted divisor in the case of a divisorial contraction.
    \end{itemize}
\end{example}
 \begin{longtable}[c]{lcllll}
     
    $(r,g,d)$&$\alpha_e$&Type & Vector & Movable & Contracted   \\ \hline
    $(2,3,4)$&$0$&Lagr & $(\shortminus1,0,0)$ & $(\shortminus1,0,0)$ &\\
    &&Flop & $(3,1,1)$ & $(0,3,4)$ & \\
    &&HC & $(2,1,1)$ & $(4,3,3)$ & $(5,3,4)$ \\ \hline
    
    $(2,6,5)$&$\frac{1}{7}$&HC & $(\shortminus1,0,0)$ & $(\shortminus10,\shortminus1,0)$ & $(\shortminus7,\shortminus1,\shortminus1)$\\
    &&Flop & $(3,1,2)$ & $(0,3,10)$ & \\
    &&Flop & $(\shortminus2,\shortminus1,\shortminus3)$ & $(10,7,20)$ & \\
    &&HC & $(4,2,5)$ & $(20,11,30)$ & $(13,7,19)$\\ \hline
    
    $(2,9,7)$&$\frac{1}{5}$&LGU & $(\shortminus1,0,0)$ & $(\shortminus8,\shortminus1,0)$ & $(\shortminus5,\shortminus1,\shortminus2)$\\
    &&Flop & $(3,1,3)$ & $(0,3,16)$ & \\
    &&Lagr & $(2,1,4)$ & $(2,1,4)$ & \\ \hline
    
    $(3,4, 6)$ & $0$&Lagr & $(\shortminus1,0,0)$ & $(\shortminus1,0,0)$ &\\
    &&Flop & $(4,1,1)$ & $(0,2,3)$ & \\
    &&LGU & $(3,1,1)$ & $(3,2,3)$ & $(5,2,3)$\\ \hline

    $(3,8,9)$&$\frac{1}{10}$&HC & $(\shortminus1,0,0)$ & $(\shortminus14,\shortminus1,0)$ & $(\shortminus10,\shortminus1,\shortminus1)$\\
    &&Flop & $(4,1,2)$ & $(0,2,7)$ & \\
    &&HC & $(\shortminus7,\shortminus3,\shortminus9)$ & $(70,29,84)$ & $(46,19,55)$\\ \hline

    $(2,27,20)$& $\frac{7}{29}$&HC&$(8, 2, 13)$&$(\shortminus104, \shortminus25, \shortminus156)$&$(\shortminus29, \shortminus7, \shortminus44)$\\
    &&Flop&$(3, 1, 9)$&$(0, 3, 52)$&\\
    &&Flop&$(5, 2, 21)$&$(52, 23, 260)$&\\
    &&HC&$(26, 11, 121)$&$(364, 155, 1716)$&$(101, 43, 476)$\\ \hline
    
    $(10,11,20)$& $0$&Lagr&$(\shortminus1, 0, 0)$&$(\shortminus1, 0, 0)$&\\
    &&Flop&$(11, 1, 1)$&$(0, 11, 20)$&\\
    &&Flop&$(10, 1, 1)$&$(20, 11, 20)$&\\
    &&Flop&$(9, 1, 1)$&$(40, 11, 20)$&\\
    &&LGU&$(\shortminus5, \shortminus1, \shortminus2)$&$(60, 11, 20)$&$(61, 11, 20)$\\
\end{longtable}
\noindent Moreover the following is an example in higher dimensions where the method is not applicable:
    \begin{longtable}[c]{llllll}
    $(r,g,d)$&$\alpha_e$&Type & Vector & Movable & Contracted   \\ \hline
    $(3,16,15)$& ?&HC&$(\shortminus15, \shortminus2, \shortminus4)$&$(\shortminus210, \shortminus29, \shortminus60)$&$(\shortminus94, \shortminus13, \shortminus27)$\\
&&Flop&$(\shortminus8, \shortminus1, \shortminus2)$&$(\shortminus60, \shortminus8, \shortminus15)$&\\
&&Flop&$(5, 1, 3)$&$(\shortminus10, \shortminus1, 0)$&\\
&&Flop&$(4, 1, 4)$&$(0, 2, 15)$&\\
&&HC&$(3, 1, 5)$&$(30, 11, 60)$&$(14, 5, 27)$\\
\end{longtable}

\printbibliography
\end{document}